\documentclass[12pt]{article}
\usepackage{amsmath,enumerate,amsfonts,amssymb,color,graphicx,amsthm}

\usepackage{hyperref}
\usepackage{todonotes}
\usepackage[normalem]{ulem}

\usepackage{cite}

\def\RR{{\mathbb R}}

\def\Sym{{\rm Sym}}

\newtheorem{theorem}{Theorem}[section]
\newtheorem{proposition}[theorem]{Proposition}
\newtheorem{corollary}[theorem]{Corollary}
\newtheorem{lemma}[theorem]{Lemma}
\newtheorem{definition}[theorem]{Definition}

\def\vareps{\varepsilon}

\def\tr{\textrm{tr}}

\DeclareFontFamily{OT1}{rsfs}{}
\DeclareFontShape{OT1}{rsfs}{m}{n}{ <-7> rsfs5 <7-10> rsfs7 <10-> rsfs10}{}
\DeclareMathAlphabet{\mycal}{OT1}{rsfs}{m}{n}

\def\calU{{\mycal U}}
\def\calK{{\mycal K}}

\def\eps{{\varepsilon}}

\begin{document}
\title{Towards a Liouville theorem for continuous viscosity solutions to fully nonlinear elliptic equations in conformal geometry}
\author{YanYan Li \thanks{Department of Mathematics, Rutgers University, Hill Center, Busch Campus, 110 Frelinghuysen Road, Piscataway, NJ 08854, USA. Email: yyli@math.rutgers.edu.}~\thanks{Partially supported by NSF Grant DMS-1501004.}~ and Luc Nguyen \thanks{Mathematical Institute and St Edmund Hall, University of Oxford, Andrew Wiles Building, Radcliffe Observatory Quarter, Woodstock Road, Oxford OX2 6GG, UK. Email: luc.nguyen@maths.ox.ac.uk.}~ and Bo Wang \thanks{School of Mathematics and Statistics, Beijing Institute of Technology, Beijing 100081, China. Email: wangbo89630@bit.edu.cn.}~\thanks{Partially supported by NNSF (11701027) and Beijing Institute of Technology Research Fund Program for Young Scholars.}}

\date{}

\maketitle

\centerline{\it Dedicated to Gang Tian on his 60th birthday with friendship.}

\begin{abstract}
We study entire continuous viscosity solutions to fully nonlinear elliptic equations involving the conformal Hessian. We prove the strong comparison principle and Hopf Lemma for (non-uniformly) elliptic equations when one of the competitors is $C^{1,1}$. We obtain as a consequence a Liouville theorem for entire solutions which are approximable by $C^{1,1}$ solutions on larger and larger compact domains, and, in particular, for entire $C^{1,1}_{\rm loc}$ solutions: they are either constants or standard bubbles.

\end{abstract}

\tableofcontents
\section{Introduction} 

It is of interest to prove Liouville theorems for entire continuous viscosity solutions of a fully nonlinear elliptic equation of the form
\begin{equation}
f(\lambda(A^u)) = 1,\qquad \lambda(A^u) \in \Gamma,\qquad u > 0 \text{ on } \RR^n,
	\label{Eq:21XI18-EE}
\end{equation}
where the conformal Hessian $A^u$ of $u$ is defined for $n \geq 3$ by
\[
A^u = -\frac{2}{n-2} u^{-\frac{n+2}{n-2}} \nabla^2 u  + \frac{2n}{(n-2)^2} u^{-\frac{2n}{n-2}} \nabla u \otimes \nabla u - \frac{2}{(n-2)^2} u^{-\frac{2n}{n-2}} |\nabla u|^2 I,
\]
$I$ is the $n \times n$ identity matrix, $\lambda(A^u)$ denotes the eigenvalues of $A^u$, $\Gamma$ is an open subset of $\RR^n$ and $f \in C^0(\bar \Gamma)$. (See \cite{Li09-CPAM}, or Definition  \ref{Def:PsiViscositySolution} below with $\psi = -\ln u$, for the definition of viscosity solutions as well as sub- and super-solutions.) Typically, $(f,\Gamma)$ is assumed to satisfy the following structural conditions.
\begin{enumerate}[(i)]
\item $(f,\Gamma)$ is symmetric, i.e.
\begin{equation} 
\text{if $\lambda \in \Gamma$ and $\lambda'$ is a permutation of $\lambda$, then $\lambda' \in \Gamma$ and $f(\lambda') = f(\lambda)$.}
	\label{Eq:02I19-Sym}
\end{equation}

\item $(f,\Gamma)$ is elliptic, i.e. 
\begin{equation}
\text{if $\lambda \in \Gamma$ and $\mu \in \bar\Gamma_n$, then $\lambda + \mu  \in \Gamma$ and $f(\lambda + \mu) \geq f(\lambda)$},
	\label{Eq:02I19-Ellipticity}
\end{equation}
where $\Gamma_n := \{\mu \in \RR^n: \mu_i > 0\}$ is the positive cone.

\item $(f,\Gamma)$ is locally strictly elliptic, i.e. for any compact subset $K$ of $\Gamma$, there is some constant $\delta(K) > 0$ such that
\begin{equation}
f(\lambda + \mu) - f(\lambda) \geq \delta(K) |\mu| \text{ for all } \lambda \in K, \mu \in \bar\Gamma_n.
	\label{Eq:02I19-LSEll}
\end{equation}

\item $f$ is locally Lipschitz, i.e. for any compact subset $K$ of $\Gamma$, there is some constant $C(K) > 0$ such that
\begin{equation}
|f(\lambda') - f(\lambda)| \leq C(K) |\lambda' - \lambda| \text{ for all } \lambda, \lambda' \in K.
	\label{Eq:02I19-LLip}
\end{equation}

\item The $1$-superlevel set of $f$ stays in $\Gamma$, namely
\begin{equation}
f^{-1}([1,\infty)) \subset \Gamma .
	\label{Eq:02I19-1level}
\end{equation}

\item $\Gamma$ satisfies
\begin{equation}
\Gamma \subset \Gamma_1 := \{\mu \in \RR^n: \mu_1 + \ldots \mu_n > 0\}.
	\label{Eq:02I19-Gam1}
\end{equation}
\end{enumerate}
It should be noted that equation \eqref{Eq:21XI18-EE} is not necessarily uniformly elliptic and that we do not assume that $\Gamma$ be convex nor $f$ be concave.

Standard examples of $(f,\Gamma)$ satisfying \eqref{Eq:02I19-Sym}-\eqref{Eq:02I19-Gam1} are given by $(f,\Gamma) = (\sigma_k^{1/k}, \Gamma_k)$, $1 \leq k \leq n$, where $\sigma_k$ is the $k$-th elementary symmetric function and $\Gamma_k$ is the connected component of $\{\lambda \in \RR^n: \sigma_k(\lambda) > 0\}$ containing the positive cone $\Gamma_n$.

Liouville theorems for \eqref{Eq:21XI18-EE} have been studied extensively. We mention here earlier results of Gidas, Ni and Nirenberg \cite{G-N-N-1981}, Caffarelli, Gidas and Spruck \cite{CGS} in the semi-linear case, of Viaclovsky \cite{Viac00-Duke,Viac00-TrAMS} for the $\sigma_k$-equations for $C^2$ solutions which are regular at infinity, of Chang, Gursky and Yang \cite{CGY03-IP} for the $\sigma_2$-equation in four dimensions, of Li and Li \cite{LiLi03, LiLi05} for $C^2$ solutions, and of Li and Nguyen \cite{LiNgSymmetry} for continuous viscosity solutions which are approximable by $C^2$ solutions on larger and larger compact domains. 

The key use of the $C^2$ regularity in the proof of the Liouville theorem in \cite{LiNgSymmetry} is the strong comparison principle and Hopf Lemma for \eqref{Eq:21XI18-EE}. In fact, if the strong comparison principle and Hopf Lemma can be established for $C^{1,\alpha}$  solutions ($0 \leq \alpha \leq 1$), a Liouville theorem is then proved in $C^{1,\alpha}$ regularity by the same arguments.

The present note is an exploration in the above direction. We establish the strong comparison principle and Hopf Lemma when one competitor is $C^{1,1}$, and obtain as a consequence a Liouville theorem in this regularity.

\begin{theorem}[Strong comparison principle]\label{Thm:SCPu}
Let $\Omega$ be an open, connected subset of $\RR^n$, $n \geq 3$, $\Gamma$ be a non-empty open subset of $\RR^n$ and $f\in C^0(\bar \Gamma)$ satisfying 
\eqref{Eq:02I19-Sym}-\eqref{Eq:02I19-1level}. Assume that 
\begin{enumerate}[(i)]
\item $u_1 \in USC(\Omega;[0,\infty))$ and $u_2 \in LSC(\Omega;(0,\infty])$ are a sub-solution and a super-solution to $f(\lambda(A^u)) = 1$ in $\Omega$ in the viscosity sense, respectively,
\item and that $u_1 \leq u_2$ in $\Omega$.
\end{enumerate}
If one of $\ln u_1$ and $\ln u_2$ belongs to $C^{1,1}_{\rm loc}(\Omega)$, then either $u_1 \equiv u_2$ in $\Omega$ or $u_1 < u_2$ in $\Omega$.
\end{theorem}

\begin{theorem}[Hopf Lemma]\label{Thm:Hopfu}
Let $\Omega$ be an open subset of $\RR^n$, $n \geq 3$, such that $\partial\Omega$ is $C^2$ near some point $\hat x \in \partial\Omega$, $\Gamma$ be a non-empty open subset of $\RR^n$ and $f\in C^0(\bar \Gamma)$ satisfying 
\eqref{Eq:02I19-Sym}-\eqref{Eq:02I19-1level}. Assume that 
\begin{enumerate}[(i)]
\item $u_1 \in USC(\Omega \cup \{\hat x\};[0,\infty))$ and $u_2 \in LSC(\Omega \cup \{\hat x\};(0,\infty])$ are a sub-solution and a super-solution to $f(\lambda(A^u)) = 1$ in $\Omega$ in the viscosity sense, respectively,
\item and that $u_1 < u_2$ in $\Omega$, and $u_1(\hat x) = u_2(\hat x)$.
\end{enumerate}
If one of $\ln u_1$ and $\ln u_2$ belongs to $C^{1,1}(\Omega \cup \{\hat x\})$, then
\[
\liminf_{s \rightarrow 0^+} \frac{(u_2 - u_1)(\hat x - s \nu(\hat x))}{s} > 0,
\]
where $\nu(\hat x)$ is the outward unit normal to $\partial\Omega$ at $\hat x$.
\end{theorem}

Our proof of the strong comparison principle and Hopf Lemma uses ideas in Caffarelli, Li and Nirenberg \cite{CafLiNir11} and an earlier work of the authors \cite{LiNgWang}. In fact we establish them for more general equations of the form
\[
F(x, \psi, \nabla \psi, \nabla^2 \psi) = 1.
\]
See Section \ref{Sec:SCPHL}, Theorem \ref{Thm:SCPpsi} and Theorem \ref{Thm:Hopfpsi}.

There has been a lot of studies on the (strong) comparison principle and Hopf Lemma for elliptic equations in related contexts. See for instance \cite{AmendolaGaliseVitolo13-DIE, BardiDaLio99-AM, HBDolcettaPorretaRosi15-JMPA, BirindelliDemengel04-AFSTM, BirindelliDemengel07-CPAA, BirGalIshii18-AIHP,  BirGalIshii-preprint, CabreCaffBook, CafLiNir11, UserGuide, DolcettaVitolo07-MC, DolcettaVitolo18-IMRN, HartmanNirenberg, HarveyLawsonSurvey2013, Ishii89-CPAM, IshiiLions90-JDE,Jensen88-ARMA, KawohlKutev98-AM, KawohlKutev00-FE, KawohlKutev07-CPDE, KoikeKosugi15-CPAA, KoikeLey11-JMAA, Li07-ARMA, Li09-CPAM, LiMonticelli12-JDE, LiNgWang, LiNir-misc,  LiWang18-ActaMSSB, LiWang-ATA, Trudinger88-RMI, Wang17-NA} and the references therein.

As mentioned earlier, a combination of the above strong comparison principle and Hopf Lemma and the proof of \cite[Theorem 1.1]{LiNgSymmetry} give the following Liouville theorem.

\begin{theorem}[Liouville theorem]\label{Thm:Liouville}
Assume that $n \geq 3$ and $(f,\Gamma)$ satisfies \eqref{Eq:02I19-Sym}-\eqref{Eq:02I19-Gam1}. Suppose that there exist $v_k \in C^{1,1}(B_{R_k}(0))$, $R_k \rightarrow \infty$, such that $f(\lambda(A^{v_k})) = 1$, $\lambda(A^{v_k}) \in \Gamma$ in the ball $B_{R_k}(0)$ of radius $R_k$ in the viscosity sense, $v_k$ converges uniformly on compact subsets of $\RR^n$ to some function $v > 0$. Then 
\medskip

either (i) $v$ is identically constant, $0 \in \Gamma$ and $f(0) = 1$, 
\medskip

or (ii) $v$ has the form
\begin{equation}
v(x) = \Big(\frac{a}{1 + b^2|x - x_0|^2}\Big)^{\frac{n-2}{2}}
	\label{Eq:02I19-vForm}
\end{equation}
for some $x_0 \in \RR^n$ and some $a,b > 0$ satisfying $f(2b^2a^{-2}, \ldots, 2b^2a^{-2}) = 1$.

\end{theorem}

It is a fact that if $u$ is $C^{1,1}$ in some open set $\Omega$, $u$ satisfies $f(\lambda(A^{u})) = 1$ in the viscosity sense in $\Omega$ if and only if it satisfies $f(\lambda(A^{u})) = 1$ almost everywhere in $\Omega$. See e.g. Lemma \ref{Lem:C11Vis}.

It should be clear that if $0 \in \Gamma$ and $f(0) = 1$, then, by \eqref{Eq:02I19-Ellipticity} and \eqref{Eq:02I19-LSEll}, $(t, \ldots, t) \in \Gamma$ and $f(t,\ldots, t) > 1$ for all $t > 0$. Hence if some constant is a solution of \eqref{Eq:21XI18-EE}, then all entire solutions of \eqref{Eq:21XI18-EE} are constant, and likewise if some function of the form \eqref{Eq:02I19-vForm} is a solution of \eqref{Eq:21XI18-EE}, then all entire solutions of \eqref{Eq:21XI18-EE} are of the form \eqref{Eq:02I19-vForm}.

An immediate consequence is:
\begin{corollary}
Assume that  $n \geq 3$ and $(f,\Gamma)$ satisfies \eqref{Eq:02I19-Sym}-\eqref{Eq:02I19-Gam1}. If $v \in C^{1,1}_{\rm loc}(\RR^n)$ is a viscosity solution of \eqref{Eq:21XI18-EE}, then $v$ is either a constant or of the form \eqref{Eq:02I19-vForm}.
\end{corollary}

The rest of the paper contains two sections. In Section \ref{Sec:SCPHL}, we state and prove our strong comparison principle and Hopf Lemma for a class of elliptic equations which is more generalized than $f(\lambda(A^u)) = 1$. In Section \ref{Sec:Liou}, we prove the Liouville theorem (Theorem \ref{Thm:Liouville}).


\section{The strong comparison principle and the Hopf Lemma}\label{Sec:SCPHL}

In this section we prove the strong comparison principle and the Hopf Lemma for elliptic equations of the form
\begin{equation}
F(x, \psi, \nabla \psi, \nabla^2 \psi) = 1 \text{ in } \Omega
	\label{Eq:24XI18-E1}
\end{equation}
where $\Omega$ is an open subset of $\RR^n$, $n \geq 1$, $F \in C(\bar\calU)$, $\calU$ is a non-empty open subset of $\bar\Omega \times \RR \times \RR^n \times \Sym_n$, and $(F,\calU)$ satisfies the following conditions.
\begin{enumerate}[(i)]

\item $(F,\calU)$ is elliptic, i.e. for all $(x,s,p,M) \in \calU, N \in \Sym_n, N \geq 0$,
\begin{equation}
(x,s,p,M + N) \in \calU \text{ and } F(x,s,p,M + N) \geq F(x,s,p,M) .
	\label{Eq:24XI18-GEllipticity}
\end{equation}
Here and below we write $N \geq 0$ for a non-negative definite matrix $N$.

\item For $x \in \bar\Omega$, let $\calU_x := \{(s,p,M) \in  \RR \times \RR^n \times \Sym_n : (x,s,p,M) \in \calU\}$. Then, 
for $x \in \bar\Omega$, the $1$-superlevel set of $F(x, \cdot)$ stays in $ \calU_x$, i.e. 
\begin{equation}
F(x,s,p,M)  < 1 \text{ for all $x \in \bar\Omega$ and $(s,p,M) \in \partial \calU_x$},
	\label{Eq:24XI18-G1level}
\end{equation}
or, equivalently,
\begin{equation*}
\{(s,p,M) \in \bar\calU_x: F(x,s,p,M)  \geq 1\} \subset \calU_x.
\end{equation*}

\item $(F,\calU)$ is locally strictly elliptic, i.e. 
for any compact subset $\calK$ of $\calU$, there is some constant $\delta = \delta(\calK) > 0$ such that, for all $(x,s,p,M) \in \calK, N \in \Sym_n, N \geq 0$,
\begin{equation}
F(x,s,p,M + N) - F(x, s, p,M) \geq \delta(\calK) |N| .
	\label{Eq:24XI18-GLSEll}
\end{equation}

\item $F$ satisfies a local Lipschitz condition with respect to $(s,p,M)$, namely for every compact subset $\calK$ of $\calU$, there there exists $C(\calK) > 0$ such that, for all $(x,s,p,M), (x,s',p',M') \in \calK$,
\begin{equation}
|F(x, s, p, M) - F(x, s', p', M')| \leq C(\calK)(|s - s'| + |p - p'| + |M - M'|).
	\label{Eq:24XI18-GLip}
\end{equation}

\end{enumerate}

To keep the notation compact, we abbreviate
\[
J_2[\psi] = (\psi, \nabla \psi, \nabla^2 \psi) \in \RR \times \RR^n \times \Sym_n.
\]

We note that equation \eqref{Eq:21XI18-EE} can be put in the form \eqref{Eq:24XI18-E1} by writing $\psi = - \ln u$, $F(J_2[\psi]) = f(\lambda(A^u))$.

To dispel confusion, we remark that $\calU$ is defined as a subset of $\bar\Omega \times \RR \times \RR^n \times \Sym_n$ rather than that of $\Omega \times \RR \times \RR^n \times \Sym_n$. In particular, the `local' properties in (iii)-(iv) are local with respect to the $(s,p,M)$-variables and not the $x$-variables.

Let us start with the definition of classical and viscosity (sub-/super-)solutions. For this we only need the ellipticity condition \eqref{Eq:24XI18-GEllipticity} and the following condition which is weaker than \eqref{Eq:24XI18-G1level}:
\begin{enumerate}[(i')]
\setcounter{enumi}{1}
\item There holds
\begin{equation}
F(x,s,p,M)  \leq 1 \text{ for all $x \in \bar\Omega$ and $(s,p,M) \in \partial \calU_x$}.
	\label{Eq:24XI18-wG1level}
\end{equation}
or, equivalently,
\begin{equation*}
\{(s,p,M) \in \bar\calU_x: F(x,s,p,M)  > 1\} \subset \calU_x.
\end{equation*}
\end{enumerate}

\begin{definition}[Classical (sub-/super-)solutions]\label{Def:PsiClassSolution}
Let $\Omega\subset\mathbb{R}^{n}$, $n \geq 1$, be an open set, and $\calU$ be a non-empty open subset of $\bar\Omega \times \RR \times \RR^n \times \Sym_n$ and $F \in C^0(\bar\calU)$ satisfying \eqref{Eq:24XI18-GEllipticity} and \eqref{Eq:24XI18-wG1level}. For a function  $\psi \in C^2(\Omega)$, we say that 
\begin{equation*}
F(x, J_2[\psi]) \leq 1 \quad \left(F(x, J_2[\psi]) \geq 1 \text{ resp.} \right) \quad\mbox{ classically in }\Omega
\end{equation*}
if there holds
\[
\text{ either } (x,J_2[\psi](x)) \notin \bar\calU \text{ or } F(x, J_2[\psi](x)) \leq 1 \text{ for all } x \in \Omega
\]
\[
\left( \; (x,J_2[\psi](x)) \in \bar\calU \text{ and } F(x, J_2[\psi](x))  \geq 1 \text{ for all } x \in \Omega\text{ resp.} \;\right).
\]

We say that a function $\psi \in C^2(\Omega)$ is a classical solution of \eqref{Eq:24XI18-E1} in $\Omega$ if we have that $(x, J_2[\psi](x)) \in \bar \calU$ and $F(x, J_2[\psi](x)) = 1$ for every $x \in\Omega$.

When $F(x, J_2[\psi]) \leq 1$ ($F(x, J_2[\psi]) \geq 1$, resp.) in $\Omega$, we also say interchangeably that $u$ is a super-solution (sub-solution) to \eqref{Eq:24XI18-E1} in $\Omega$.
\end{definition}

In the above definition, the role of condition \eqref{Eq:24XI18-wG1level} is manifested in the property that if $\psi_k$ is a sequence of super-solutions which converges in $C^2$ to some $\psi$, then $\psi$ is also a super-solution. When discussing only sub-solutions, condition \eqref{Eq:24XI18-wG1level} can be dropped.

\begin{definition}[Viscosity (sub-/super-)solutions]\label{Def:PsiViscositySolution}
Let $\Omega\subset\mathbb{R}^{n}$, $n \geq 1$, be an open set, and $\calU$ be a non-empty open subset of $\bar\Omega \times \RR \times \RR^n \times \Sym_n$ and $F \in C^0(\bar\calU)$ satisfying \eqref{Eq:24XI18-GEllipticity} and \eqref{Eq:24XI18-wG1level}. For a function  $\psi \in LSC(\Omega; \RR \cup \{\infty\})$ ($\psi \in USC(\Omega; \RR \cup \{-\infty\})$ resp.), we say that 
\begin{equation*}
F(x, J_2[\psi]) \leq 1 \quad \left(F(x, J_2[\psi]) \geq 1 \text{ resp.} \right) \quad\mbox{in }\Omega
\end{equation*}
in the viscosity sense if for any $x_{0}\in\Omega$, $\varphi\in C^{2}(\Omega)$, $(\psi-\varphi)(x_{0})=0$ and 
\begin{equation*}
\psi-\varphi\geq0\quad(\psi-\varphi\leq0  \text{ resp.})\quad\mbox{near }x_{0},
\end{equation*}
there holds
\[
\text{ either } (x_0,J_2[\varphi](x_0)) \notin \bar\calU \text{ or } F(x_0, J_2[\varphi](x_0)) \leq 1
\]
\[
\left( \; (x_0,J_2[\varphi](x_0)) \in \bar\calU \text{ and } F(x_0, J_2[\varphi](x_0))  \geq 1 \text{ resp.} \;\right).
\]

We say that a function $\psi \in C^0(\Omega)$ satisfies \eqref{Eq:24XI18-E1} in the viscosity sense in $\Omega$ if we have both that $F(x, J_2[\psi]) \geq 1$ and $F(x, J_2[\psi]) \leq 1$ in $\Omega$ in the viscosity sense.

When $F(x, J_2[\psi]) \leq 1$ ($F(x, J_2[\psi]) \geq 1$, resp.) in $\Omega$ in the viscosity sense, we also say interchangeably that $u$ is a viscosity super-solution (sub-solution) to \eqref{Eq:24XI18-E1} in $\Omega$.
\end{definition}

The main results in this section are the following.

\begin{theorem}[Strong comparison principle]\label{Thm:SCPpsi}
Let $\Omega$ be an open, connected subset of $\RR^n$, $n \geq 1$, $\calU$ be a non-empty open subset of $\bar\Omega \times \RR \times \RR^n \times \Sym_n$ and $F\in C^0(\bar \calU)$ satisfying 
\eqref{Eq:24XI18-GEllipticity}-\eqref{Eq:24XI18-GLip}. Assume that 
\begin{enumerate}[(i)]
\item $\psi_1 \in USC(\Omega;\RR \cup\{-\infty\})$ and $\psi_2 \in LSC(\Omega;\RR \cup \{\infty\})$ are a sub-solution and a super-solution to \eqref{Eq:24XI18-E1} in $\Omega$ in the viscosity sense, respectively,
\item and that $\psi_1 \leq \psi_2$ in $\Omega$.
\end{enumerate}
If one of $\psi_1$ and $\psi_2$ belongs to $C^{1,1}_{\rm loc}(\Omega)$, then either $\psi_1 \equiv \psi_2$ in $\Omega$ or $\psi_1 < \psi_2$ in $\Omega$.
\end{theorem}

\begin{theorem}[Hopf Lemma]\label{Thm:Hopfpsi}
Let $\Omega$ be an open subset of $\RR^n$, $n \geq 1$, such that $\partial\Omega$ is $C^2$ near some point $\hat x \in \partial\Omega$, $\calU$ be a non-empty open subset of $\bar\Omega \times \RR \times \RR^n \times \Sym_n$ and $F\in C^0(\bar \calU)$ satisfying 
\eqref{Eq:24XI18-GEllipticity}-\eqref{Eq:24XI18-GLip}. Assume that 
\begin{enumerate}[(i)]
\item $\psi_1 \in USC(\Omega \cup\{\hat x\};\RR \cup\{-\infty\})$ and $\psi_2 \in LSC(\Omega \cup\{\hat x\};\RR \cup \{\infty\})$ are a sub-solution and a super-solution to \eqref{Eq:24XI18-E1} in $\Omega$ in the viscosity sense, respectively,
\item and that $\psi_1 < \psi_2$ in $\Omega$, and $\psi_1(\hat x) = \psi_2(\hat x)$.
\end{enumerate}
If one of $\psi_1$ and $\psi_2$ belongs to $C^{1,1}(\Omega \cup\{\hat x\})$, then
\[
\liminf_{s \rightarrow 0^+} \frac{(\psi_2 - \psi_1)(\hat x - s \nu(\hat x))}{s} > 0,
\]
where $\nu(\hat x)$ is the outward unit normal to $\partial\Omega$ at $\hat x$.
\end{theorem}

If $\psi_1$ and $\psi_2$ are continuous and one of them is $C^2$, the above theorems were proved in Caffarelli, Li, Nirenberg \cite{CafLiNir11}.

Before turning to the proof of the above theorems, we give some simple statements for viscosity solutions.

\begin{lemma}\label{Lem:C11Vis}
Let $\Omega\subset\mathbb{R}^{n}$, $n \geq 1$, be an open set, and $\calU$ be a non-empty open subset of $\bar\Omega \times \RR \times \RR^n \times \Sym_n$ and $F \in C^0(\bar\calU)$ satisfying \eqref{Eq:24XI18-GEllipticity} and \eqref{Eq:24XI18-wG1level}. Suppose that $\psi$ is semi-concave (semi-convex resp.) in $\Omega$, then
\begin{equation*}
F(x, J_2[\psi]) \leq 1 \quad \left(F(x, J_2[\psi]) \geq 1 \text{ resp.} \right) \quad\mbox{in }\Omega \text{ in the viscosity sense} 
\end{equation*}
if and only if
\begin{equation*}
\text{ either } (x, J_2[\psi](x)) \notin \bar \calU \text{ or } F(x, J_2[\psi](x)) \leq 1  \quad\mbox{a.e. in }\Omega
\end{equation*}
\begin{equation*}
\left(\;(x, J_2[\psi](x)) \in \bar \calU \text{ and } F(x, J_2[\psi](x)) \geq 1 \quad\mbox{a.e. in }\Omega   \text{ resp.}  \;\right).
\end{equation*}
\end{lemma}

Recall that $\psi$ is semi-concave (semi-convex resp.) in $\Omega$ if there is some $K > 0$ such that $\psi - \frac{K}{2}|x|^2$ ($\psi + \frac{K}{2}|x|^2$ resp.) is locally concave (convex resp.) in $\Omega$. By a theorem of Alexandrov, Buselman and Feller (see e.g. \cite[Theorem 1.5]{CabreCaffBook}), semi-concave (or semi-convex) functions are almost everywhere punctually second order differentiable.

\begin{proof}
(a) Consider the inequality $F(x, J_2[\psi]) \leq 1$. 

Since $\psi$ is semi-concave, it is almost everywhere punctually second order differentiable. Suppose that $F(x, J_2[\psi]) \leq 1$ in $\Omega$ in the viscosity sense and $x_0$ is a point where $\psi$ is punctually second order differentiable. Then we can use
\[
\varphi(x) = \psi(x_0) + \nabla \psi(x_0) \cdot (x - x_0) + (x - x_0)^T \nabla^2 \psi(x_0) (x - x_0) - \delta|x - x_0|^2
\]
for any $\delta > 0$ as test functions at $x_0$ to see that 
\[
\text{ either } (x_0, J_2[\psi](x_0) - (0,0,2\delta I )) \notin \bar \calU \text{ or } F(x_0, J_2[\psi](x_0) - (0,0,2\delta I )) \leq 1.
\]
Sending $\delta \rightarrow 0$ and using \eqref{Eq:24XI18-wG1level}, we obtain
\[
\text{ either } (x_0, J_2[\psi](x_0)) \notin \bar \calU \text{ or } F(x_0, J_2[\psi](x_0) ) \leq 1.
\]

Conversely, assume that either $(x, J_2[\psi](x)) \notin \bar \calU$ or $F(x, J_2[\psi](x)) \leq 1$ for almost all $x \in \Omega$, and suppose, for some $x_{0}\in\Omega$ and $\varphi\in C^{2}(\Omega)$, that $(\psi -\varphi)(x_{0})=0$ and $\psi-\varphi\geq0 $ near $x_{0}$. We need to show that
\[
\text{ either } (x_0, J_2[\varphi](x_0)) \notin\bar\calU \text{ or } F(x_0, J_2[\varphi](x_0))  \leq 1.
\]
If $(x_0, J_2[\varphi](x_0)) \notin\calU$, we are done by \eqref{Eq:24XI18-wG1level}. We assume henceforth that $(x_0, J_2[\varphi](x_0)) \in\calU$.

Replacing $\varphi$ by $\varphi - \delta|x - x_0|^2$ for some small $\delta > 0$ and letting $\delta \rightarrow 0$ eventually, we may assume without loss of generality that
\[
\psi > \varphi \text{ in } B_{2r_0}(x_0)\setminus \{x_0\} \subset \Omega \text{ for some } r_0 > 0.
\]

For small $\eta > 0$, let $\xi = \xi_\eta = (\psi - \varphi - \eta)^-$ and let $\Gamma_\xi$ be the concave envelop of $\xi$ in $B_{2r_0}(x_0)$. We have by \cite[Lemma 3.5]{CabreCaffBook} that
\[
\int_{\{\xi = \Gamma_{\xi}\}} \det(-\nabla^2 \Gamma_{\xi}) \geq \frac{1}{C} (\sup_{B_{2r_0}(x_0)} \xi)^n > 0.
\]
In particular, the set $\{\xi = \Gamma_{\xi}\}$ has non-zero measure. Thus, we can find $y_\eta \in \{\xi = \Gamma_{\xi}\}$ such that $\psi$ is punctually second order differentiable at $y_\eta$, either $(y_\eta, J_2[\psi](y_\eta)) \notin \bar\calU$ or $F(y_\eta, J_2[\psi](y_\eta)) \leq 1$ and
\begin{align}
0 > \xi(y_\eta) &= \psi(y_\eta) - \varphi(y_\eta) - \eta  \geq - \eta
	,\label{Eq:04I19-C1a}\\
|\nabla\xi(y_\eta)| &= |\nabla \psi(y_\eta) - \nabla\varphi(y_\eta)| \leq C\eta
	,\label{Eq:04I19-C1b}\\
\nabla^2\xi(y_\eta) &= \nabla^2 \psi(y_\eta) - \nabla^2 \varphi(y_\eta) \geq 0
	.\label{Eq:04I19-C1c}
\end{align}
Recalling that $(x_0, J_2[\varphi](x_0)) \in\calU$ and noting that $y_\eta \rightarrow x_0$ as $\eta \rightarrow 0$, we deduce from \eqref{Eq:24XI18-GEllipticity} and \eqref{Eq:04I19-C1a}-\eqref{Eq:04I19-C1c} that, for all small $\eta$, $(y_\eta, J_2[\varphi](y_\eta))$, $(y_\eta, J_2[\psi](y_\eta))$ and $(y_\eta, \psi(y_\eta),  \nabla \psi(y_\eta), \nabla^2 \varphi(y_\eta))$ belong to $\calU$. We then have
\begin{eqnarray*}
1 
	&\geq& F(y_\eta, J_2[\psi](y_\eta))\\
	&\stackrel{\eqref{Eq:24XI18-GEllipticity},\eqref{Eq:04I19-C1c}}{\geq}& F(y_\eta, \psi(y_\eta),  \nabla \psi(y_\eta), \nabla^2 \varphi(y_\eta))\\
	&\stackrel{\eqref{Eq:04I19-C1a}, \eqref{Eq:04I19-C1b}}{\geq}& F(y_\eta, \varphi(y_\eta),  \nabla \varphi(y_\eta), \nabla^2 \varphi(y_\eta)) + o_\eta(1),
\end{eqnarray*}
where $o_\eta(1) \rightarrow 0$ as $\eta \rightarrow 0$ and where we have used the uniform continuity of $F$ on compact subsets of $\bar\calU$. Letting $\eta \rightarrow 0$, we obtain the assertion.

(b) Consider now the inequality $F(x, J_2[\psi]) \geq 1$. This case is treated similarly, but is slightly easier as we do not have a dichotomy in the almost everywhere sense.

Since $\psi$ is semi-convex, it is almost everywhere punctually second order differentiable. If $F(x, J_2[\psi]) \geq 1$ is satisfied in the viscosity sense, then, as in the previous case, if $x_0$ is a point where $\psi$ is punctually second order differentiable, then
\[
(x_0, J_2[\psi](x_0) + (0,0,2\delta)) \in \bar\calU \text{ and } F(x_0, J_2[\psi](x_0) + (0,0,2\delta)) \geq 1 \text{ for any } \delta > 0,
\]
and so, upon sending $\delta \rightarrow 0$, we obtain
\[
(x_0, J_2[\psi](x_0) ) \in \bar\calU \text{ and } F(x_0, J_2[\psi](x_0) ) \geq 1.
\]

Suppose that $F(x, J_2[\psi](x)) \geq 1$ holds almost everywhere in $\Omega$ and suppose, for some $x_{0}\in\Omega$ and $\varphi\in C^{2}(\Omega)$, that $(\psi -\varphi)(x_{0})=0$ and $\psi-\varphi\leq0 $ near $x_{0}$. We need to show that
\[
F(x_0, J_2[\varphi](x_0))  \geq 1.
\]

Replacing $\varphi$ by $\varphi + \delta|x - x_0|^2$ for some small $\delta > 0$ and letting $\delta \rightarrow 0$ eventually, we may assume without loss of generality that
\[
\psi < \varphi \text{ in } B_{2r_0}(x_0)\setminus \{x_0\} \subset \Omega \text{ for some } r_0 > 0.
\]

For small $\eta > 0$, let $\xi = \xi_\eta = (\psi - \varphi + \eta)^+$ and let $\Gamma_\xi$ be the concave envelop of $\xi$ in $B_{2r_0}(x_0)$. We have by \cite[Lemma 3.5]{CabreCaffBook} that
\[
\int_{\{\xi = \Gamma_{\xi}\}} \det(-\nabla^2 \Gamma_{\xi}) \geq \frac{1}{C} (\sup_{B_{2r_0}(x_0)} \xi)^n > 0.
\]
In particular, the set $\{\xi = \Gamma_{\xi}\}$ has positive measure. Thus, we can find $y_\eta \in \{\xi = \Gamma_{\xi}\}$ such that $\psi$ is punctually second order differentiable at $y_\eta$, $F(y_\eta, J_2[\psi](y_\eta)) \geq 1$ and
\begin{align}
0 < \xi(y_\eta) &= \psi(y_\eta) - \varphi(y_\eta) + \eta  \leq \eta
	,\label{Eq:03I19-C1a}\\
|\nabla\xi(y_\eta)| &= |\nabla \psi(y_\eta) - \nabla\varphi(y_\eta)| \leq C\eta
	,\label{Eq:03I19-C1b}\\
\nabla^2\xi(y_\eta) &= \nabla^2 \psi(y_\eta) - \nabla^2 \varphi(y_\eta) \leq 0
	.\label{Eq:03I19-C1c}
\end{align}
It follows that
\begin{eqnarray*}
1 
	&\leq& F(y_\eta, J_2[\psi](y_\eta))\\
	&\stackrel{\eqref{Eq:24XI18-GEllipticity},\eqref{Eq:03I19-C1c}}{\leq}& F(y_\eta, \psi(y_\eta),  \nabla \psi(y_\eta), \nabla^2 \varphi(y_\eta))\\
	&\stackrel{\eqref{Eq:03I19-C1a}, \eqref{Eq:03I19-C1b}}{\leq}& F(y_\eta, \varphi(y_\eta),  \nabla \varphi(y_\eta), \nabla^2 \varphi(y_\eta)) + o_\eta(1),
\end{eqnarray*}
where $o_\eta(1) \rightarrow 0$ as $\eta \rightarrow 0$ and where we have used the uniform continuity of $F$ on compact subsets of $\bar\calU$. Letting $\eta \rightarrow 0$ and noting that $y_\eta \rightarrow x_0$, we conclude the proof.
\end{proof}

\subsection{Proof of the strong comparison principle}

We first prove the strong comparison principle for subsolutions and $C^{1,1}$ strict super-solutions.

\begin{proposition}\label{Prop:WSCPpsi}
Let $\Omega$ be an open, connected subset of $\RR^n$, $n \geq 1$, $\calU$ be a non-empty open subset of $\bar\Omega \times \RR \times \RR^n \times \Sym_n$ and $F\in C^0(\bar \calU)$ satisfying \eqref{Eq:24XI18-GEllipticity}-\eqref{Eq:24XI18-G1level}. Assume that 
\begin{enumerate}[(i)]
\item $\psi_1 \in USC(\Omega;\RR \cup\{-\infty\})$ satisfies%
\[
F(x,J_2[\psi_1]) \geq 1 \text{ in $\Omega$ in the viscosity sense},
\]
\item $\psi_2 \in C^{1,1}_{\rm loc}(\Omega)$ satisfies for some constant $a < 1$,
\[
\text{ either } (x,J_2[\psi_2](x)) \notin\bar\calU  \text{ or } F(x,J_2[\psi_2](x)) \leq a \qquad \text{ a.e. in $\Omega$},
\]
\item $\psi_1 \leq \psi_2$ in $\Omega$ and  $\psi_1 < \psi_2$ near $\partial\Omega$.
\end{enumerate}
Then $\psi_1 < \psi_2$ in $\Omega$.
\end{proposition}

\begin{proof} We follow \cite{LiNgWang}. Assume by contradiction that there exists some $\hat x \in \Omega$ such that $\psi_1(\hat x) = \psi_2(\hat x)$.

\medskip
\noindent\underline{Step 1:} We regularize $\psi_1$ using sup-convolution. 
\medskip

This step is well known, see e.g. \cite[Chapter 5]{CabreCaffBook}.

Take some bounded domain $A$ containing $\hat x$ such that $\bar A \subset \Omega$ and $\psi_1 < \psi_2$ on $\partial A$.

We define, for small $\eps > 0$ and $x \in A$,
\[
\hat \psi_\eps(x) = \sup_{y \in \Omega} \Big(\psi_1(y) -\frac{1}{\eps}|x - y|^2\Big).
\]

It is well-known that $\hat \psi_\eps \geq \psi_1$, $\hat \psi_\eps$ is semi-convex, $\nabla^2 \hat\psi_\eps \geq -\frac{2}{\vareps}I$ a.e. in $A$, and $\hat\psi_\eps$ converges monotonically to $\psi_1$ as $\eps \rightarrow 0$. Furthermore, for every $x \in A$, there exists $x^* = x^*(\eps, x)$ such that 
\begin{equation}
\hat\psi_\eps(x) =  \psi_1(x^*) -\frac{1}{\eps}|x - x^*|^2.
	\label{Eq:22XI18-A3}
\end{equation}

We note that if $x$ is a point where $\hat\psi_\eps$ is punctually second order differentiable, then $\psi_1$ `can be touched from above' at $x^*$ by a quadratic polynomial:
\begin{equation}
\psi_1(x^{*}+z)
	\leq \hat\psi_\eps(x) + \frac{1}{\eps}|x^{*}-x|^{2} + \nabla \hat\psi_{\eps}(x)\cdot z
	+\frac{1}{2}z^{T}\nabla^{2}\hat\psi_{\eps}(x)z 
	  +o(|z|^{2})\quad\mbox{ as }z\rightarrow 0,
	\label{Eq:22XI18-A4}
\end{equation}
which is a consequence of the inequalities
\begin{align*}
\hat\psi_{\eps}(x+z) 
	&\leq \hat\psi_{\eps}(x)
	+ \nabla \hat\psi_{\eps}(x)\cdot z
	+\frac{1}{2}z^{T}\nabla^{2}\hat\psi_{\eps}(x)z 
	+ o(|z|^{2}),\quad\mbox{ as }z\rightarrow 0,\\
\hat\psi_{\eps}(x+z) 
	&\geq \psi_1(x^{*}+z) - \frac{1}{\eps}|x^{*}-x|^{2}
	.
\end{align*}
(Here we have used the definition of $\hat\psi_{\eps}$ in the last inequality.)

An immediate consequence of \eqref{Eq:22XI18-A3}-\eqref{Eq:22XI18-A4} and the fact that $\psi_1$ is a sub-solution of \eqref{Eq:24XI18-E1} is that 
\begin{equation}
F(x^*,\hat\psi_\eps(x) + \frac{1}{\eps}|x^{*}-x|^{2}, \nabla\hat\psi_\eps(x), \nabla^2\hat\psi_\eps(x)) \geq 1.
	\label{Eq:22XI18-A2}
\end{equation}

\medskip
\noindent \underline{Step 2:} We proceed to derive a contradiction as in \cite{LiNgWang}.
\medskip

For small $\eta > 0$, let $\tau = \tau(\eps,\eta)$ be such that
\[
\eta = \sup_A (\hat\psi_\eps - \psi_2 + \tau).
\]
Then 
\begin{align}
\tau 
	&= \psi_1(\hat x) - \psi_2(\hat x) + \tau \leq \hat\psi_\eps(\hat x) - \psi_2(\hat x) + \tau \leq \eta,
		\label{Eq:04I19-M1}\\
\tau
	&= \eta - \sup_A (\hat\psi_\eps - \psi_2)
			\geq \eta - \sup_A (\hat\psi_\eps - \psi_1).
		\label{Eq:04I19-M2}
\end{align} 

Suppose that $\eps$ and $\eta$ are sufficiently small so that $\xi  := \hat\psi_\eps - \psi_2 + \tau$ is negative on $\partial A$. Let $\Gamma_{\xi^+}$ denote the concave envelop of $\xi^+ = \max(\xi,0)$. Since $\xi$ is semi-convex and $\xi \leq 0$ on $\partial A$, we have by \cite[Lemma 3.5]{CabreCaffBook} that
\[
\int_{\{\xi = \Gamma_{\xi^+}\}} \det(-\nabla^2 \Gamma_{\xi^+}) \geq \frac{1}{C(\Omega)} (\sup_\Omega \xi)^n > 0.
\]
In particular, the set $\{\xi = \Gamma_{\xi^+}\}$ has positive measure. Recall that $\hat \psi_\eps$ and $\psi_2$ is almost everywhere punctually second order differentiable, we can find $y = y_{\eps,\eta} \in \{\xi = \Gamma_{\xi^+}\}$ such that $\hat \psi_\eps$ and $\psi_2$ are punctually second order differentiable at $y$, $|J_2[\psi_2](y)| \leq C \|\psi\|_{C^{1,1}(\bar A)}$, either $(y, J_2[\psi_2](y)) \notin \bar\calU$ or $F(y, J_2[\psi_2](y)) \leq a$, and
\begin{align}
0 < \xi(y) &= \hat\psi_\eps(y) - \psi_2(y) + \tau  \leq \eta
	,\label{Eq:23XI18-C1a}\\
|\nabla\xi(y)| &= |\nabla \hat\psi_\eps(y) - \nabla \psi_2(y)| \leq C\eta
	,\label{Eq:23XI18-C1b}\\
\nabla^2\xi(y) &= \nabla^2 \hat\psi_\eps(y) - \nabla^2 \psi_2(y) \leq 0
	.\label{Eq:23XI18-C1c}
\end{align}

We claim that
\begin{equation}
	\liminf_{\eps \rightarrow 0} \frac{1}{\eps} |y^* - y|^2 \leq \eta, 
	\label{Eq:28XI18-Cl2}
\end{equation}
where $y^* = x^*(\eps, y)$ and $x^*$ is defined in \eqref{Eq:22XI18-A3}.

Let us assume \eqref{Eq:28XI18-Cl2} for now and go on with the proof. From, \eqref{Eq:24XI18-GEllipticity}, \eqref{Eq:22XI18-A2},  \eqref{Eq:23XI18-C1c}, we have $(y^*,\hat\psi_\eps(y) + \frac{1}{\eps}|y^{*}-y|^{2}, \nabla\hat\psi_\eps(y), \nabla^2 \psi_2(y)) \in \bar\calU$ and 
\begin{eqnarray}
1 
	&\stackrel{\eqref{Eq:22XI18-A2}}{\leq}& F(y^*,\hat\psi_\eps(y) + \frac{1}{\eps}|y^{*}-y|^{2}, \nabla\hat\psi_\eps(y), \nabla^2\hat\psi_\eps(y))
	\nonumber\\
	&\stackrel{\eqref{Eq:24XI18-GEllipticity}, \eqref{Eq:23XI18-C1c}}{\leq}& F(y^*,\hat\psi_\eps(y) + \frac{1}{\eps}|y^{*}-y|^{2}, \nabla\hat\psi_\eps(y), \nabla^2 \psi_2(y)).
		\label{Eq:10I19-M1}
\end{eqnarray}

By the boundedness of $J_2[\psi_2](y)$, we may assume that
\begin{equation}
(y, J_2[\psi_2](y))= (y_{\eps,\eta}, J_2[\psi_2](y_{\eps,\eta})) \rightarrow (y_0,p_0) \text{ along a sequence $\eps, \eta \rightarrow 0$}.
	\label{Eq:10I19-M2}
\end{equation}
By \eqref{Eq:04I19-M1}, \eqref{Eq:04I19-M2}, \eqref{Eq:23XI18-C1a} and \eqref{Eq:23XI18-C1b}, we then have
\[
(y^*,\hat\psi_\eps(y) + \frac{1}{\eps}|y^{*}-y|^{2}, \nabla\hat\psi_\eps(y), \nabla^2 \psi_2(y)) \rightarrow (y_0, p_0).
\]
Thus by \eqref{Eq:24XI18-G1level} and \eqref{Eq:10I19-M1}, $(y_0, p_0) \in \calU$ and $F(y_0, p_0) \geq 1$. But this implies, in view of \eqref{Eq:10I19-M2}, that $(y, J_2[\psi_2](y)) \in \calU$ along a sequence $\eps, \eta \rightarrow 0$ and so
\[
1 \leq F(y_0,p_0) = \lim_{\eps,\eta\rightarrow 0} F(y, J_2[\psi_2](y)) \leq a,
\]
which is a contradiction.

To conclude the proof, it remains to establish \eqref{Eq:28XI18-Cl2}.

\medskip
\noindent\underline{Proof of \eqref{Eq:28XI18-Cl2}:} Suppose for some $\eta$ and some sequence $\eps_m \rightarrow 0$  that $\frac{1}{\eps_m} |y_m^* - y_m|^2 \rightarrow d$ where $y_m := y_{\eps_m, \eta}$ and $y_m^* := y_{\eps_m, \eta}^*$. (Note that $\frac{1}{\eps}|x^* - x|^2 \leq C$, so this assumption makes sense.) We need to show that $d \leq \eta$. 

Let $\tau_m = \tau( \eps_m, \eta)$. Without loss of generality, we assume further that $y_m \rightarrow y_0$ and $\tau_m \rightarrow \tau_0$. By the convergence of $y_m$ and of $\frac{1}{\eps_m} |y_m^* - y_m|^2$, we have that $y_m^* \rightarrow y_0$. Thus, by the upper semi-continuity of $\psi_1$, we have
\[
\limsup_{m \rightarrow \infty} \psi_1(y_m^*) \leq \psi_1(y_0).
\]
Hence, by \eqref{Eq:22XI18-A3}, \eqref{Eq:04I19-M1} and the left half of \eqref{Eq:23XI18-C1a}, we have
\begin{eqnarray*}
0
	&\leq& \limsup_{m \rightarrow \infty} \frac{1}{\eps_m} |y_m^* - y_m|^2
		\stackrel{ \eqref{Eq:22XI18-A3}}{=} \limsup_{m \rightarrow \infty} (\psi_1(y_m^*) - \hat\psi_{\eps_m}(y_m))\\
	&\stackrel{\eqref{Eq:23XI18-C1a}}{\leq}& \limsup_{m \rightarrow \infty} (\psi_1(y_m^*) - \psi_2(y_m) + \tau_m)\\
	&\stackrel{\eqref{Eq:04I19-M1}}{\leq}&  \psi_1(y_0) - \psi_2(y_0) + \eta
		= \lim_{m \rightarrow \infty} (\hat\psi_{\eps_m}(y_0) - \psi_2(y_0)) + \eta\\
	&\leq& \lim_{m \rightarrow \infty} \sup_{A} (\hat\psi_{\eps_m} - \psi_2) + \eta
		\leq \sup_{A} (\psi_1 - \psi_2) + \eta
		= \eta.
\end{eqnarray*}
This proves \eqref{Eq:28XI18-Cl2} and concludes the proof.
\end{proof}

By analogous arguments, we have:
\begin{proposition}\label{Prop:WSCPpsiX}
Let $\Omega$ be an open, connected subset of $\RR^n$, $n \geq 1$, $\calU$ be a non-empty open subset of $\bar\Omega \times \RR \times \RR^n \times \Sym_n$ and $F\in C^0(\bar \calU)$ satisfying 
\eqref{Eq:24XI18-GEllipticity} and \eqref{Eq:24XI18-wG1level}.
Assume that 
\begin{enumerate}[(i)]
\item $\psi_1 \in C^{1,1}_{\rm loc}(\Omega;\RR)$ and $\psi_2 \in LSC(\Omega \cup \{\infty\})$ satisfy for some constant $a' > 1$,
\[
F(x,J_2[\psi_1]) \geq a' \text{ and } F(x,J_2[\psi_2]) \leq 1 \text{ in $\Omega$ in the viscosity sense},
\]

\item $\psi_1 \leq \psi_2$ in $\Omega$ and  $\psi_1 < \psi_2$ near $\partial\Omega$.
\end{enumerate}
Then $\psi_1 < \psi_2$ in $\Omega$.
\end{proposition}

\begin{proof} We argue as in the proof of Proposition \ref{Prop:WSCPpsi}, exchanging the roles of $\psi_1$ and $\psi_2$ and sup-convolution and inf-convolution.

Assume by contradiction that there exists some $\hat{x}\in\Omega$ such that $\psi_{1}(\hat{x})=\psi_{2}(\hat{x})$.

\medskip
\noindent\underline{Step 1:} We regularize $\psi_{2}$ by using inf-convolution.
\medskip 

Take some bounded domain $A$ containing $\hat{x}$ such that $\bar{A}\subset\Omega$ and $\psi_{1}<\psi_{2}$ on $\partial A$.

We define, for small $\varepsilon>0$ and $x\in A$, 
 \begin{equation*}
 \hat{\psi}^{\varepsilon}(x)=\inf\limits_{y\in\Omega}\left(\psi_{2}(y)+\frac{1}{\varepsilon}|x-y|^{2}\right).
 \end{equation*}
 
 It is well-known that $\hat{\psi}^{\varepsilon}\leq \psi_{2}$, $\hat{\psi}^{\varepsilon}$ is semi-concave, $\nabla^{2}\hat{\psi}^{\varepsilon}\leq\frac{2}{\varepsilon}I$ a.e. in $A$, and $\hat{\psi}^{\varepsilon}$ converges monotonically to $\psi_{2}$ as $\varepsilon\rightarrow0$. Furthermore, for every $x\in A$, there exists $x_{*}=x_{*}(\varepsilon,x)$ such that 
 \begin{equation}
 \hat{\psi}^{\varepsilon}(x)=\psi_{2}(x_{*})+\frac{1}{\varepsilon}|x-x_{*}|^{2}.\label{eq21'}
  \end{equation}
  
  We note that if $x$ is a point where $\hat{\psi}^{\varepsilon}$ is punctually second order differentiable, then $\psi_{2}$ `can be touched from below' at $x_{*}$ by a quadratic polynomial:
  \begin{equation}
  \psi_{2}(x_{*}+z)\geq\hat{\psi}^{\varepsilon}(x)-\frac{1}{\varepsilon}|x_{*}-x|^{2}+\nabla\hat{\psi}^{\varepsilon}(x)\cdot z+\frac{1}{2}z^{T}\nabla^{2}\hat{\psi}^{\varepsilon}(x)z+o(|z|^{2}),\quad\mbox{as }z\rightarrow0,\label{eq22'}
  \end{equation}
  which is a consequence of the inequalities 
  \begin{align*}
 \hat{\psi}^{\varepsilon}(x+z)&\geq \hat{\psi}^{\varepsilon}(x)+\nabla\hat{\psi}^{\varepsilon}(x)\cdot z+\frac{1}{2}z^{T}\nabla^{2}\hat{\psi}^{\varepsilon}(x)z+o(|z|^{2}),\quad\mbox{as }z\rightarrow0,\\
 \hat{\psi}^{\varepsilon}(x+z)&\leq \psi_{2}(x_{*}+z)+\frac{1}{\varepsilon}|x_{*}-x|^{2}.
   \end{align*}
   (Here we have used the definition of $\hat{\psi}^{\varepsilon}$ in the last inequality.)
   
   An immediate consequence of (\ref{eq21'})-(\ref{eq22'}) and the fact that $\psi_{2}$ is a super-solution of (9) is that either   \begin{equation}\label{add1}
   (x_{*},\hat{\psi}^{\varepsilon}(x)-\frac{1}{\varepsilon}|x_{*}-x|^{2},\nabla\hat{\psi}^{\varepsilon}(x),\nabla^{2}\hat{\psi}^{\varepsilon}(x))\notin\bar{\calU} ,
   \end{equation}
   or 
   \begin{equation}
   F(x_{*},\hat{\psi}^{\varepsilon}(x)-\frac{1}{\varepsilon}|x_{*}-x|^{2},\nabla\hat{\psi}^{\varepsilon}(x),\nabla^{2}\hat{\psi}^{\varepsilon}(x))\leq1.\label{add2}
   \end{equation}

\medskip
\noindent\underline{Step 2:} We proceed to derive a contradiction as in \cite{LiNgWang}.
\medskip

For small $\eta>0$, let $\tau=\tau(\varepsilon,\eta)$ be such that 
\begin{equation*}
\eta=\sup\limits_{A}(\psi_{1}-\hat{\psi}^{\varepsilon}+\tau).
\end{equation*}
Then 
\begin{align}
\tau&=\psi_{1}(\hat{x})-\psi_{2}(\hat{x})+\tau\leq \psi_{1}(\hat{x})-\hat{\psi}^{\varepsilon}(\hat{x})+\tau\leq\eta,\label{eq24'}\\
\tau&=\eta-\sup\limits_{A}(\psi_{1}-\hat{\psi}^{\varepsilon})\geq \eta-\sup\limits_{A}(\psi_{2}-\hat{\psi}^{\varepsilon}).\label{eq25'}
\end{align}

Suppose that $\varepsilon$ and $\eta$ are sufficiently small so that $\xi:=\psi_{1}-\hat{\psi}^{\varepsilon}+\tau$ is negative on $\partial A$. Let $\Gamma_{\xi^{+}}$ denote the concave envelop of $\xi^{+}=\max\{\xi,0\}$. Since $\xi$ is semi-convex and $\xi\leq0$ on $\partial A$, we have by [8,Lemma 3.5] that 
 \begin{equation*}
 \int_{\{\xi=\Gamma_{\xi^{+}}\}}\det(-\nabla^{2}\Gamma_{\xi^{+}})\geq\frac{1}{C(\Omega)}(\sup\limits_{\Omega}\xi)^{n}>0.
 \end{equation*}
In particular, the set $\{\xi=\Gamma_{\xi^{+}}\}$ has positive measure. Recall that $\hat{\psi}^{\varepsilon}$ and $\psi_{1}$ are almost everywhere punctually second order differentiable, we can find $y=y_{\varepsilon,\eta}\in\{\xi=\Gamma_{\xi^{+}}\}$ such that $\hat{\psi}^{\varepsilon}$ and $\psi_{1}$ are punctually second order differentiable at $y$, $|J_{2}[\psi_{1}](y)|\leq C\|\psi_{1}\|_{C^{1,1}(\bar{A})}$, 
\begin{align}
0<\xi(y)&=\psi_{1}(y)-\hat{\psi}^{\varepsilon}(y)+\tau\leq\eta,\label{eq26'}\\
|\nabla\xi(y)|&=|\nabla\psi_{1}(y)-\nabla\hat{\psi}^{\varepsilon}(y)|\leq C\eta,\label{eq27'}\\
\nabla^{2}\xi(y)&=\nabla^{2}\psi_{1}(y)-\nabla^{2}\hat{\psi}^{\varepsilon}(y)\leq0,\label{eq28'}
\end{align}
and 
\begin{equation}
(y,J_{2}[\psi_{1}](y))\in \bar{\calU},~~F(y,J_{2}[\psi_{1}](y))\geq a'.\label{add3}
\end{equation}

We claim that 
\begin{equation}
\liminf\limits_{\varepsilon\rightarrow0}\frac{1}{\varepsilon}|y_{*}-y|^{2}\leq\eta,\label{eq29'}
\end{equation}
where $y_{*}=x_{*}(\varepsilon,y)$ and $x_{*}$ is defined in (\ref{eq21'}).

Let us assume (\ref{eq29'}) for now and go on with the proof. As in Case 1, we may assume that $(y,J_{2}[\psi_{1}](y)) \rightarrow (y_0, p_0)$ as $\eps, \eta \rightarrow 0$. By \eqref{add3}, $F(y_0,p_0) \geq a'$ and so by \eqref{Eq:24XI18-wG1level}, $(y_0,p_0) \in \calU$. Also, by \eqref{eq24'}, \eqref{eq25'}, \eqref{eq26'} and \eqref{eq27'}, 
\[
\big(y_{*},\hat{\psi}^{\varepsilon}(y)-\frac{1}{\varepsilon}|y_{*}-y|^{2},\nabla\hat{\psi}^{\varepsilon}(y),\nabla^{2}\psi_{1}(y)\big) \rightarrow (y_0, p_0) \text{ as } \eps, \eta \rightarrow 0,
\]
and so
\[
\big(y_{*},\hat{\psi}^{\varepsilon}(y)-\frac{1}{\varepsilon}|y_{*}-y|^{2},\nabla\hat{\psi}^{\varepsilon}(y),\nabla^{2}\psi_{1}(y)\big) \in \calU \text{ along a sequence } \eps, \eta \rightarrow 0,
\]
Now, we have by \eqref{Eq:24XI18-GEllipticity} and \eqref{eq28'} that \eqref{add2} holds at $x = y$ and so
   \begin{eqnarray*}
   1&\stackrel{\eqref{add2}}{\geq}& F(y_{*},\hat{\psi}^{\varepsilon}(y)-\frac{1}{\varepsilon}|y_{*}-y|^{2},\nabla\hat{\psi}^{\varepsilon}(y),\nabla^{2}\hat{\psi}^{\varepsilon}(y))\\
   &\stackrel{\eqref{eq28'}}{\geq}& F(y_{*},\hat{\psi}^{\varepsilon}(y)-\frac{1}{\varepsilon}|y_{*}-y|^{2},\nabla\hat{\psi}^{\varepsilon}(y),\nabla^{2}\psi_{1}(y)) \\
    &=& F(y,J_{2}[\psi_{1}](y))+o_{\varepsilon,\eta}(1)\\
   &\stackrel{\eqref{add3}}{\geq}& a'+o_{\varepsilon,\eta}(1),
   \end{eqnarray*}
where $\lim\limits_{\varepsilon,\eta\rightarrow0}o_{\varepsilon,\eta}(1)=0$ and where we have used the (local uniform) continuity of $F$ in the second-to-last equality. This gives a contradiction as $a'>1$.

To conclude the proof, it remains to establish (\ref{eq29'}).

\medskip
\noindent\underline{Proof of (\ref{eq29'}):} Suppose for some $\eta>0$ and some sequence $\varepsilon_{m}\rightarrow0$ that $\frac{1}{\varepsilon_{m}}|(y_{m})_{*}-y_{m}|^{2}\rightarrow d$ where $y_{m}:=y_{\varepsilon_{m},\eta}$ and $(y_{m})_{*}:=(y_{\varepsilon_{m},\eta})_{*}$. (Note that $\frac{1}{\varepsilon}|x_{*}-x|^{2}\leq C$, so this assumption makes sense.) We need to show that $d\leq\eta$.         
 
 Let $\tau_{m}=\tau(\varepsilon_{m},\eta)$. Without loss of generality, we assume further that $y_{m}\rightarrow y_{0}$ and $\tau_{m}\rightarrow\tau_{0}$. By the convergence of $y_{m}$ and of $\frac{1}{\varepsilon_{m}}|(y_{m})_{*}-y_{m}|^{2}$, we have that $(y_{m})_{*}\rightarrow y_{0}$. Thus, by the lower semi-continuity of $\psi_{2}$, we have 
 \begin{equation*}
 \liminf\limits_{m\rightarrow\infty}\psi_{2}((y_{m})_{*})\geq\psi_{2}(y_{0}).
 \end{equation*}
 Hence, by (\ref{eq21'}), (\ref{eq24'}) and the left half of (\ref{eq26'}), we have 
 \begin{eqnarray*}
 0&\leq& \liminf\limits_{m\rightarrow\infty}\frac{1}{\varepsilon_{m}}|(y_{m})_{*}-y_{m}|^{2} \stackrel{(\ref{eq21'})}{=}\liminf\limits_{m\rightarrow\infty}\big(\hat{\psi}^{\varepsilon_{m}}(y_{m})-\psi_{2}((y_{m})_{*})\big)\\
 & \stackrel{(\ref{eq26'})}{\leq} & \liminf\limits_{m\rightarrow\infty}(\psi_{1}(y_{m})-\psi_{2}((y_{m})_{*})+\tau_{m})\\
 & \stackrel{(\ref{eq24'})}{\leq}&\psi_{1}(y_{0})-\psi_{2}(y_{0})+\eta=\lim\limits_{m\rightarrow\infty}(\psi_{1}(y_{0})-\hat{\psi}^{\varepsilon_{m}}(y_{0}))+\eta\\
 &\leq& \lim\limits_{m\rightarrow\infty}\sup\limits_{A}(\psi_{1}-\hat{\psi}^{\varepsilon_{m}})+\eta\leq \sup\limits_{A}(\psi_{1}-\psi_{2})+\eta=\eta.
         \end{eqnarray*}
         This proves (\ref{eq29'}) and concludes the proof.
         
\end{proof}

We now give the

\begin{proof}[Proof of Theorem \ref{Thm:SCPpsi}] Arguing by contradiction, suppose the conclusion is wrong, then we can find a closed ball $\bar B \subset \Omega$ of radius $R > 0$ and a point $\hat x \in \partial B$ such that
\[
\psi_1 < \psi_2 \text{ in } \bar B \setminus \{\hat x\} \text{ and } \psi_1(\hat x) = \psi_2(\hat x). 
\]
Without loss of generality, we assume the center of $B$ is the origin.

\medskip
\noindent\underline{Case 1:} Consider first the case $\psi_2$ is $C^{1,1}$. 
\medskip

In the proof, $C$ denotes some generic constant which may vary from lines to lines but depends only on an upper bound for $\|\psi_2\|_{C^{1,1}(\bar\Omega)}$, $\Omega$ and $(F,\calU)$.

In view of Proposition \ref{Prop:WSCPpsi}, it suffices to deform $\psi_2$ to a strict super-solution $\tilde\psi_2$ in some open ball $A$ around $\hat x$ such that $\tilde\psi_2 > \psi_1$ on $\partial A$ and $\inf_{A} (\tilde\psi_2 - \psi_1) = 0$.
We adapt the argument in  \cite{CafLiNir11}, which assumes that $\psi_2$ is $C^2$.

Using that $\psi_2$ is $C^{1,1}$, a theorem of Alexandrov, Buselman and Feller (see e.g. \cite[Theorem 1.5]{CabreCaffBook}) and Lemma \ref{Lem:C11Vis}, we can find some $\Lambda > 0$ and a set $Z$ of zero measure such that $\psi_2$ is punctually second order differentiable in $\Omega \setminus Z$,
\begin{equation}
|J_2[\psi_2]|   \leq \Lambda \text{ in } \Omega \setminus Z.
	\label{Eq:05I19-V2}
\end{equation}
and
\begin{equation}
\text{ either } (x,J_2[\psi_2](x)) \notin \bar\calU \text{ or } F(x,J_2[\psi_2](x)) \leq 1 \text{ in } \Omega \setminus Z.
	\label{Eq:05I19-V1}
\end{equation}

By \eqref{Eq:24XI18-G1level}, there is some small constant $\theta_0 > 0$
\[
F(x,s,p,M) \leq 1 - 2\theta_0 \text{ for all } x\in \bar\Omega, (s,p,M) \in \partial \calU_x,  |s| + |p| + |M| \leq \Lambda + 2.
\]
Hence
\[
\calK := \Big\{(x,s,p,M) \in \calU: F(x,s,p,M) \geq 1 - \theta_0 , x \in \bar\Omega,
	 |s| + |p| + |M| \leq \Lambda + 1\Big\} 
\]
and
\begin{multline*}
\calK' := \Big\{(x,s,p,M) \in \calU: F(x,s,p,M) \geq 1 - \theta_0/2 , x \in \bar\Omega,\\
	 |s| + |p| + |M| \leq \Lambda + 1/2\Big\} \subset \calK 
\end{multline*}
are compact.

For $\alpha > 1$, $\mu > 0$ and $\tau > 0$ which will be fixed later, let 
\begin{align}
E(x) 
	&= E_\alpha(x)
	= e^{-\alpha|x |^2},\nonumber\\
h(x) 
	&= h_\alpha(x)
	= e^{-\alpha|x |^2} - e^{-\alpha R^2},\nonumber\\
\zeta(x) 
	&= \zeta_\alpha(x) = \cos(\alpha^{1/2}(x_1 - \hat x_1)),\nonumber\\
\tilde \psi_{\mu,\tau} 
	&=  \psi_2 - \mu \,(h - \tau) \zeta.
	\label{Eq:05I19-X1}
\end{align}
Let $A$ be a ball centered at $\hat x$ such that $\zeta > \frac{1}{2}$ in $A$ and $\tau_0 = \sup_A h > 0$. 

It is clear that, for $0 \leq \tau \leq \tau_0$ and all sufficiently small $\mu$,
\begin{equation*}
\text{$\tilde\psi_{\mu, \tau} > \psi_1$ on $\partial A$}.
\end{equation*}

We compute
\begin{align*}
\nabla \tilde\psi_{\mu,\tau}(x)
	&= \nabla \psi_2(x) + 2\mu\,\alpha\,E\,\zeta\,x + \mu \alpha^{1/2} (h - \tau) \sin(\alpha^{1/2} (x_1 - \hat x_1)) e_1,\\
\nabla^2 \tilde\psi_{\mu,\tau}(x)
	&= \nabla^2 \psi_2(x)
		- 2\mu \alpha E\,\zeta (2\alpha\,x \otimes x  -  I)\\
		&\qquad - 2\mu\,\alpha^{3/2} \,E\, \sin(\alpha^{1/2} (x_1 - \hat x_1)(e_1 \otimes x + x \otimes e_1) \\
		&\qquad + \mu \alpha (h - \tau) \zeta e_1 \otimes e_1.
\end{align*}
We thus have
\[
J_2[\tilde\psi_{\mu,\tau}](x) = J_2[\psi_2](x) - \big(0,0,4\mu\,\alpha^2\,E \zeta \,x \otimes  x + \mu \tau \alpha \zeta e_1 \otimes e_1\big) + O(\mu (\alpha^{3/2} E + \alpha^{1/2}\tau)).
\]

Now if $x \in A \setminus Z$ is such that
\begin{multline*}
(x,J_2[\tilde\psi_{\mu,\tau}](x)), (x,J_2[\psi_2](x))\\
	 \text{and } (x,J_2[\tilde\psi_{\mu,\tau}] + (0,0,4\mu\,\alpha^2\,E \zeta \,x \otimes  x + \mu \tau \alpha \zeta e_1 \otimes e_1) \text{ lie in $\calK$},
\end{multline*}
then
\begin{align*}
&F(x,J_2[\psi_2]) + C\mu\alpha^{3/2} E + C\mu \tau \alpha^{1/2}\\
	&\qquad \stackrel{\eqref{Eq:24XI18-GLip}}{\geq} F\big(x,J_2[\tilde\psi_{\mu,\tau}] + (0,0,4\mu\,\alpha^2\,E \zeta \,x \otimes  x + \mu \tau \alpha \zeta e_1 \otimes e_1)\big)\\
	&\qquad \stackrel{\eqref{Eq:24XI18-GLSEll}}{\geq} F\big(x,J_2[\tilde\psi_{\mu,\tau}]) + \frac{1}{C} \mu\,\alpha^2\,E  + \frac{1}{C}\mu \tau \alpha,
\end{align*}
and so, by selecting a sufficiently large $\alpha$, we thus obtain for some $\beta > 0$ and all sufficiently small $\mu$,
\begin{align}
F\big(x,J_2[\tilde\psi_{\mu,\tau}])
	&\leq F\big(x,J_2[\psi_2]) - \beta \mu \stackrel{\eqref{Eq:05I19-V1}}{\leq}  1 - \beta \mu.
		\label{Eq:05I19-W1}
\end{align}

Now for every $x \in A \setminus Z$ satisfying
$J_2[\tilde\psi_{\mu,\tau}](x) \in \calU_x$ and $F(x,J_2[\tilde\psi_{\mu,\tau}](x)) \geq 1 - \theta_0/2$,
we have, in view of \eqref{Eq:05I19-V2}, that $|J_2[\tilde\psi_{\mu,\tau}](x)| \leq \Lambda + 1/2$ and so $(x,J_2[\tilde\psi_{\mu,\tau}](x))$ lies in $\calK'$ for all small $\mu$. By squeezing $\mu$ further, we then have that $(x, J_2[\psi_2](x))$ and $(x,J_2[\tilde\psi_{\mu,\tau}] + (0,0,4\mu\,\alpha^2\,E \zeta \,x \otimes  x + \mu \tau \alpha \zeta e_1 \otimes e_1)$  lie in $\calK$. In particular, \eqref{Eq:05I19-W1} holds.

Taking $\tilde\beta = \min (\beta, \frac{\theta_0}{2\mu})$, we thus obtain that
\begin{equation*}
\text{either } J_2[\tilde\psi_{\mu,\tau}](x) \notin \bar\calU_x \text{ or } F(x,J_2[\tilde\psi_{\mu,\tau}](x)) \leq 1 - \tilde \beta \mu\text{ in } A\setminus Z.
\end{equation*}
Noting that
\begin{equation*}
 \inf_A(  \tilde\psi_{\mu, 0} - \psi_1) \leq 0 \leq \inf_A(  \tilde\psi_{\mu, \tau_0} - \psi_1),
\end{equation*}
we can select $\tau_1 \in [0,\tau_0]$ such that
\[
\inf_A(  \tilde\psi_{\mu, \tau_1} - \psi_1) = 0.
\]
The desired $\tilde\psi_2$ is taken to be $\tilde \psi_{\mu,\tau_1}$. The conclusion follows from Proposition \ref{Prop:WSCPpsi}.

\medskip
\noindent\underline{Case 2:} Consider now the case $\psi_1$ is $C^{1,1}$.
\medskip

The proof is similar. $C$ will now denote some generic constant which depends only on an upper bound for $\|\psi_1\|_{C^{1,1}(\bar\Omega)}$, $\Omega$ and $(F,\calU)$.

In view of Proposition \ref{Prop:WSCPpsiX}, it suffices to deform $\psi_{1}$ to a strict sub-solution $\tilde{\psi}_{1}$ in some open ball $A$ around $\hat{x}$ such that $\psi_{2}>\tilde{\psi}_{1}$ on $\partial A$ and $\inf\limits_{A}(\psi_{2}-\tilde{\psi}_{1})=0$.

Using that $\psi_{1}$ is $C^{1,1}$, a theorem of Alexandrov, Buselman and Feller and Lemma \ref{Lem:C11Vis}, we can find some $\Lambda>0$ and a set $Z$ of zero measure such that $\psi_{1}$ is puntually second order differentiable in $\Omega \setminus Z$, 
\begin{equation*}
|J_{2}[\psi_{1}]|\leq \Lambda~\mbox{in }\Omega \setminus Z,\label{eq30'}
\end{equation*}  
and, by \eqref{Eq:24XI18-G1level},
\begin{equation}
(x,J_{2}[\psi_{1}](x))\in \calU ~\mbox{and }F(x,J_{2}[\psi_{1}](x))\geq1~\mbox{in }\Omega \setminus Z.\label{eq31'}
\end{equation}

For $\alpha>1$, $\mu>0$ and $\tau>0$ which will be fixed later, let $E, h, \zeta, A, \tau_0$ be as in Case 1, and amend the definition of $\tilde\psi_{\mu,\tau}$ to
\begin{align}
\tilde{\psi}_{\mu,\tau}&=\psi_{1}+\mu(h-\tau)\zeta.\label{eq32'}
\end{align}

It is clear that, for $0\leq\tau\leq\tau_{0}$ and all sufficiently small $\mu$, 
\begin{equation*}
\tilde{\psi}_{\mu,\tau}<\psi_{2}~\mbox{on }\partial A.
\end{equation*}

As before, we have
\begin{equation*}
J_{2}[\tilde{\psi}_{\mu,\tau}](x)=J_{2}[\psi_{1}](x) +(0,0,4\mu\alpha^{2}E\zeta x\otimes x+\mu\tau \alpha\zeta e_{1}\otimes e_{1})+O(\mu(\alpha^{3/2}E+\alpha^{1/2}\tau)).
\end{equation*}

It is clear from \eqref{eq31'} that $(x,J_{2}[\psi_{1}](x))$ belongs to $\calK'$ for all $x \in A \setminus Z$. We thus have for all sufficiently small $\mu$ and $x \in A \setminus Z$ that 
\begin{align*}
&(x,J_{2}[\tilde{\psi}_{\mu,\tau}]),~(x,J_{2}[\psi_{1}])\\
&\quad\quad\quad\quad\quad\mbox{ and }(x,J_{2}[\tilde{\psi}_{\mu,\tau}]-(0,0,4\mu\alpha^{2}E\zeta x\otimes x+\mu\tau \alpha\zeta e_{1}\otimes e_{1}))\mbox{ lie in }\calK.
\end{align*}
Therefore,
\begin{eqnarray*}
\lefteqn{F(x,J_{2}[\psi_{1}]) -C\mu\alpha^{3/2}E-C\mu\tau\alpha^{1/2}} \\
&\stackrel{\eqref{Eq:24XI18-GLip}}{\leq}& F\big(x,J_{2}[\tilde{\psi}_{\mu,\tau}]-(0,0,4\mu\alpha^{2}E\zeta x\otimes x+\mu\tau \alpha\zeta e_{1}\otimes e_{1})\big)\\
&\stackrel{\eqref{Eq:24XI18-GLSEll}}{\leq}& F(x,J_{2}[\tilde{\psi}_{\mu,\tau}])-\frac{1}{C}\mu\alpha^{2}E-\frac{1}{C}\mu\tau\alpha,
\end{eqnarray*}
and so, by selecting a sufficiently large $\alpha$, we thus obtain for some $\beta>0$ and all sufficiently small $\mu$,
\begin{equation*}
F(x,J_{2}[\tilde{\psi}_{\mu,\tau}])\geq F(x,J_{2}[\psi_{1}])+\beta\mu  \stackrel{\eqref{eq31'}}{\geq} 1+\beta\mu.
\end{equation*}

Noting that 
\begin{equation*}
\inf\limits_{A}(\psi_{2}-\tilde{\psi}_{\mu,0})\leq 0\leq \inf\limits_{A}(\psi_{2}-\tilde{\psi}_{\mu,\tau_{0}}),
\end{equation*}
we can select $\tau_{1}\in[0,\tau_{0}]$ such that 
\begin{equation*}
\inf\limits_{A}(\psi_{2}-\tilde{\psi}_{\mu,\tau_{1}})=0.
\end{equation*} 
The desired $\tilde{\psi}_{1}$ is taken to be $\tilde{\psi}_{\mu,\tau_{1}}$. The conclusion follows from Proposition \ref{Prop:WSCPpsiX} (and Lemma \ref{Lem:C11Vis}).
\end{proof}

\subsection{Proof of the Hopf Lemma}

\begin{proof}[Proof of Theorem \ref{Thm:Hopfpsi}]
We will only consider the case that $\psi_2$ is $C^{1,1}$, since the case when $\psi_1$ is $C^{1,1}$ can be treated similarly.

Since $\partial\Omega$ is $C^2$ near $\hat x$, we can find a ball $B$ such that $\bar B \subset \Omega \cup \{\hat x\}$ and $\hat x \in \partial B$. Thus we may assume without loss of generality that $\Omega = B$ is a ball centered at the origin, $u_1$ and $u_2$ are defined on $\bar B$ and $u_1 < u_2$ in $\bar B \setminus \{\hat x\}$.

The function $\tilde\psi_{\mu,\tau} = \psi_2 - \mu(h - \tau)\zeta$ defined by \eqref{Eq:05I19-X1} in the proof of Theorem \ref{Thm:SCPpsi} satisfies for some open ball $A$ centered at $\hat x$, some constant $\beta > 0$, and all $0 \leq \tau \leq \tau_0 := \sup_{A \cap B} h$ that
\begin{equation}
\text{either } J_2[\tilde\psi_{\mu,\tau}](x) \notin \bar\calU_x \text{ or } F(x,J_2[\tilde\psi_{\mu,\tau}](x)) \leq 1 -  \beta\mu \text{ a.e. in } A \cap B.
	\label{Eq:03I19-N2}
\end{equation}
If $\tilde\psi_{\mu,0} \geq \psi_1$ in $A \cap B$ for some $\mu > 0$, we are done by the explicit form of $h$. Suppose otherwise that 
\[
\inf_{A \cap B} ( \tilde\psi_{\mu,0} - \psi_1) < 0 .
\]
Noting that 
\[
0 \leq \inf_{A \cap B}(  \tilde\psi_{\mu, \tau_0} - \psi_1),
\]
we can find $\tau_1 \in (0,\tau_0]$ such that 
\[
\inf_{A \cap B} ( \tilde\psi_{\mu,\tau_1} - \psi_1) = 0.
\]
Recall the definition of $h$, we have also that
\[
\inf_{\partial (A \cap B)} ( \tilde\psi_{\mu,\tau_1} - \psi_1) > 0.
\]
Recalling \eqref{Eq:03I19-N2}, we obtain a contradiction to Proposition \ref{Prop:WSCPpsi}.
\end{proof}


\section{Proof of the Liouville theorem}\label{Sec:Liou}

In this section, we prove our Liouville theorem. Let us start with some preliminary. Define
\[
U = \{M \in \Sym_n: \lambda(U) \in \Gamma\}
\]
and
\[
F(M) = f(\lambda(M)).
\]
By \eqref{Eq:02I19-Sym}-\eqref{Eq:02I19-Gam1}, we have
\begin{enumerate}[(i)]
\item $(F,U)$ is elliptic, i.e. 
\begin{equation}
\text{if } M \in U \text{ and } N \geq 0 \text{, then } M + N \in U \text{ and } F(M + N) \geq F(M).
	\label{Eq:21XI18-Ellipticity}
\end{equation}
\item $(F,U)$ is locally strictly elliptic, i.e. for any compact subset $K$ of $U$, there is some constant $\delta(K) > 0$ such that
\begin{equation}
F(M + N) - F(M) \geq \delta(K) |N| \text{ for all } M \in K, N \geq 0.
	\label{Eq:21XI18-LSEll}
\end{equation}

\item $F$ is locally Lipschitz, i.e. for any compact subset $K$ of $U$, there is some constant $C(K) > 0$ such that
\begin{equation}
|F(M') - F(M)| \leq C(K) |M' - M| \text{ for all } M, M' \in K.
	\label{Eq:21XI18-LLip}
\end{equation}

\item The $1$-superlevel set of $F$ stays in $U$, namely
\begin{equation}
F^{-1}([1,\infty)) \subset U .
	\label{Eq:21XI18-1level}
\end{equation}

\item $(F,U)$ is invariant under the orthogonal group $O(n)$, i.e. 
\begin{equation}
\text{if } M \in U \text{ and } R \in O(n) \text{, then } R^t M R  \in U \text{ and } F(R^t M R) = F(M).
	\label{Eq:21XI18-O(n)Inv}
\end{equation}

\item $U$ satisfies
\begin{equation}
\tr(M) \geq 0 \text{ for all } M \in U.
	\label{Eq:23XI18-Gam1}
\end{equation}
\end{enumerate}

From \eqref{Eq:21XI18-Ellipticity}-\eqref{Eq:21XI18-1level}, we see that the strong comparison principle (Theorem \ref{Thm:SCPpsi}) and the Hopf Lemma (Theorem \ref{Thm:Hopfpsi}) are applicable to the equation $F(A^u) = 1$ by setting $\psi = -\ln u$.

An essential ingredient for our proof is a conformal property of the conformal Hessian $A^w$, inherited from the conformal structure of $\RR^n$. Recall that a map $\varphi: \RR^n \cup\{\infty\} \rightarrow \RR^n \cup\{\infty\}$ is called a M\"obius transformation if it is the composition of finitely many translations, dilations and inversions. Now if $\varphi$ is a M\"obius transformation and if we set $w_\varphi = |J_\varphi|^{\frac{n-2}{2n}}w \circ \varphi$ where $J_\varphi$ is the Jacobian of $\varphi$, then
\[
A^{w_\varphi}(x) = O_\varphi(x)^t A^w(\varphi(x)) O_\varphi(x)
\]
for some orthogonal $n \times n$ matrix $O_\varphi(x)$. In particular, by \eqref{Eq:21XI18-O(n)Inv},
\begin{equation}
F(A^{w_\varphi}(x)) = F(A^w( \varphi(x))).
	\label{Eq:CIProp}
\end{equation}

\medskip
\begin{proof}[Proof of Theorem \ref{Thm:Liouville}] Having established the Hopf Lemma and the strong comparison principle, we can follow the proof of \cite[Theorem 1.1]{LiNgSymmetry}, which draws on ideas from \cite{LiLi05}, to reach the conclusion. We give a sketch here for readers' convenience. For details, see \cite[Section 2]{LiNgSymmetry}.

We use the method of moving spheres. For a function $w$ defined on a subset of $\RR^n$, we define
\[
w_{x,\lambda}(y) = \frac{\lambda^{n-2}}{|y - x|^{n-2}}w\Big(x + \frac{\lambda^2(y -x )}{|y -x|^2}\Big)
\]
wherever the expression makes sense.

\medskip
\noindent\underline{Step 1:} We set up the moving sphere method.
\medskip

Since $v_k$ is locally uniformly bounded, local gradient estimates (see e.g. \cite[Theorem 2.1]{LiNgSymmetry}, \cite[Theorem 1.10]{Li09-CPAM}), imply that $|\nabla v_k|$ is locally uniformly bounded and so $v_k$ converges to $v$ in $C^{0,\alpha}_{loc}(\RR^n)$ and $v \in C^{0,1}_{loc}(\RR^n)$.

We note that, by \eqref{Eq:23XI18-Gam1}, $v$ is super-harmonic. Thus, by the positivity of $v$ and the maximum principle, we have
\begin{equation}
v(y)\ge
\frac{1}{C} (1+|y|)^{2-n}  \text{ for all } y \in \RR^n,
	\label{Eq:vSuperhar}
\end{equation}
and so we may also assume without loss of generality that 
\begin{equation}
\|v_k - v\|_{C^0(B_{R_k}(0))} \leq R_k^{-n} \text{ and } 
v_k(y)\ge
\frac{1}{C} (1+|y|)^{2-n}  \text{ for all } y \in B_{R_k}(0).
	\label{Eq:vkSuperhar}
\end{equation}

Using \eqref{Eq:vkSuperhar} and the local uniform boundedness of $|\nabla v_k|$, one can show that there is a function $\lambda^{(0)}: \RR^n \rightarrow (0,\infty)$ such that for all $k$,
\begin{equation}
(v_k)_{x,\lambda} \le v_k\
\mbox{in}\ B_{R_k}(0) \setminus B_\lambda(x), \forall \ 0<\lambda<
\lambda^{(0)}(x), |x| < R_k/5.
	\label{Eq:23XI18-X1}
\end{equation}
See \cite[Lemma 2.2]{LiNgSymmetry}.

Define, for $|x| < R_k/5$,
\[
\bar \lambda_k(x)
=\sup\Big\{ 0 < \mu < R_k/5: 
u_{x,\lambda} \le u \text{ in } B_{R_k}(0) \setminus B_\lambda(x), \forall 0<\lambda<\mu\Big\}.
\]
By \eqref{Eq:23XI18-X1}, $\bar\lambda_k(x) \in [\lambda^{(0)}(x),R_k/5]$. Set
\[
\bar\lambda(x) = \liminf_{k \rightarrow \infty} \bar\lambda_k(x) \in  [\lambda^{(0)}(x),\infty].
\]
$\bar\lambda(x)$ is sometimes referred to as the moving sphere radius of $v$ at $x$,

\medskip
\noindent\underline{Step 2:} We show that if $\bar \lambda(x)<\infty$ for some $x\in \RR^n$,
then 
\begin{equation}
\alpha := \liminf_{|y|\to \infty} |y|^{n-2}u(y)  =  \lim_{|y|\to \infty}
|y|^{n-2}v_{x,\bar\lambda(x)}(y)
=\bar \lambda(x)^{n-2}v(x)
<\infty.
	\label{Eq:23XI18-X4s}
\end{equation}
(Note that $\alpha > 0$ by \eqref{Eq:vSuperhar}.)

We have
\[
(v_k)_{x,\bar \lambda_k(x)} \le v_k\
\mbox{in}\ \RR^n\setminus B_{\bar \lambda_k(x)}(x), 
\]
By the conformal invariance of the conformal Hessian \eqref{Eq:CIProp}, 
$(v_k)_{x,\bar \lambda_k(x)}$ satisfies
\[
F(A^{(v_k)_{x,\bar \lambda_k(x)}})=1
\quad \text{ in }\ \RR^n\setminus \overline{  B_{\bar \lambda_k(x)}(x) }.
\]
We can now apply the strong comparison principle (Theorem \ref{Thm:SCPpsi}) and the Hopf Lemma (Theorem \ref{Thm:Hopfpsi}) to conclude that there exists $y_k \in \partial B_{R_k}(0)$ such that $(v_k)_{x,\bar \lambda_k(x)} = v_k(y_k)$. (See the proof of \cite[Lemma 4.5]{LiLi05}.) 

It follows that
\begin{align*}
\alpha 
	&\leq \liminf_{k\to\infty} |y_k|^{n-2} v(y_k)
		=  \liminf_{k\to\infty} |y_k|^{n-2} v_k(y_k)\\
	&= \liminf_{k\to\infty} |y_k|^{n-2} (v_k)_{x,\bar \lambda_k(x)}(y_k)
		=(\bar \lambda(x))^{n-2}v(x) < \infty.
\end{align*}
The opposite inequality that $\alpha  \geq (\bar \lambda(x))^{n-2}v(x)$ is an easy consequence of the inequality $v_{ x, \bar \lambda(x) }\le v$ in $\RR^n\setminus B_{\bar \lambda(x) }(x)$. This proves \eqref{Eq:23XI18-X4s}.

\medskip
\noindent\underline{Step 3:} We show that either $v$ is constant or $\bar\lambda(x) < \infty$ for all $x \in \RR^n$.

Suppose that $\bar\lambda(x_0) = \infty$ for some $x_0$. Then we hve
\[
v_{x_0,\lambda} \leq v \text{ in } \RR^n \setminus B_\lambda(x_0) \text{ for all } \lambda > 0.
\]
It follows that, for every unit vector $e$, the function $r \mapsto r^{\frac{n-2}{2}} v(x_0 + re)$ is non-decreasing. It follows that
\[
r^{n-2} \inf_{\partial B_r(x_0)} v \geq r^{\frac{n-2}{2}} \inf_{\partial B_1(x_0)} v
\]
and so
\[
\alpha = \liminf_{|y| \rightarrow \infty} |y|^{n-2} v(y) = \infty.
\]
Thus, by Step 2 above, we have $\bar\lambda(x) = \infty$ for all $x \in \RR^n$. This implies that $v$ is constant; see e.g. \cite{Li06-JFA}, \cite[Lemma C.1]{LiNgSymmetry}. This implies that $0 \in \Gamma$ and $f(0) = 1$.

\medskip
\noindent\underline{Step 4:} By Steps 2 and 3, it remains to consider the case where, for every $x \in \RR^n$, there exists $0 < \bar\lambda(x) < \infty$ such that 
\begin{enumerate}[(i)]
\item $v_{x,\bar\lambda(x)} \leq v$ in $\RR^n \setminus B_{\bar\lambda(x)}(x)$,
\item and
\[
\alpha = \lim_{|y| \rightarrow \infty} |y|^{n-2} v(y) = \lim_{|y| \rightarrow \infty} |y|^{n-2} v_{x,\bar\lambda(x)}(y).
\]
\end{enumerate}

In some sense, we have a strong comparison principle situation where touching occurs at infinity. If $v$ was $C^{1,1}$, this would imply that $v_{x,\bar\lambda(x)} \equiv v$ and so a calculus argument would then show that $v$ has the desired form (see \cite[Lemma 11.1]{LiZhang03}).

Since we have not established the strong comparison principle in $C^{0,1}$ regularity, we resort to a different argument, which was first observed in \cite{LiLi05} for $C^2$ solution and \cite{Li07-ARMA} for $C^{0,1}$ solutions. It turns out that, (i) and (ii) together with the super-harmonicity of $v$ imply directly that there exist $a, b > 0$ and $x_0 \in \RR^n$ such that
\[
u(x) = \Big(\frac{a}{1 + b^2|x - x_0|^2}\Big)^{\frac{n-2}{2}}.
\]
See \cite{Li07-ARMA, LiNgSymmetry}. This concludes the proof.
\end{proof}

\newcommand{\noopsort}[1]{}

\end{document}